\documentclass[11pt,a4paper]{article}
\usepackage[margin=28mm]{geometry}
\usepackage{amsmath,amssymb,amsthm}
\usepackage[T1]{fontenc}
\usepackage{lmodern}
\usepackage{cite}
\usepackage{hyperref}
\hypersetup{pdfborder={0 0 0},pdftitle={On a problem of Ma and Chen concerning minimal additive complements},pdfauthor={Shi-Qiang Chen and Quan-Hui Yang}}
\newtheorem{theorem}{Theorem}[section]
\newtheorem{lemma}[theorem]{Lemma}
\theoremstyle{remark}

\title{On a problem of Ma and Chen concerning minimal additive complements\thanks{The first author is supported by the National Natural Science Foundation of China (Grant No. 12301003), the Anhui Provincial Natural Science Foundation (Grant No. 2308085QA02) and the University Natural Science Research Project of Anhui Province (Grant No. 2022AH050171). The second author is supported by the National Natural Science Foundation of China (Grants No. 12371005 and 12671008).}}
\author{Shi-Qiang Chen$^{a}$,\qquad Quan-Hui Yang$^{b,}$\thanks{Corresponding author. E-mail addresses: csq20180327@163.com (S.-Q. Chen), yangquanhui01@163.com (Q.-H. Yang)}}
\date{}
\makeatletter
\renewcommand{\@maketitle}{%
  \begin{center}
    {\Large\@title\par}%
    \ifx\@author\@empty\else\medskip{\normalsize\@author\par}\fi
    \ifx\@date\@empty\else\medskip{\normalsize\@date\par}\fi
  \end{center}}
\makeatother

\begin{document}
\maketitle
\begin{center}
\small
$a$. School of Mathematics and Statistics,\\
Anhui Normal University, Wuhu 241002, P.R. China\\[3pt]
$b$. School of Mathematical Sciences and Ministry of Education Key Laboratory for NSLSCS,\\
Nanjing Normal University, Nanjing 210023, P.R. China
\end{center}

\begin{abstract}
Let $m\geq1$ be an integer and let
$W=(m\mathbb N_0+X)\cup U\cup V\subseteq\mathbb Z$ be bounded below, where
$X\ne\varnothing$, $[U]_m\subseteq[X]_m$, $[X]_m\cap[V]_m=\varnothing$,
and $V$ contains no infinite arithmetic progression with common difference $m$.
Here $[S]_m$ denotes the set of residue classes modulo $m$ represented by $S$.
Suppose that $W$ satisfies the $\mathcal C$-condition modulo $m$.
For $0\leq r<m$, let $V_r=V\cap(r+m\mathbb Z)$, and put
$V'=\bigcup_{|V_r|=\infty}V_r$, where the union is over $0\leq r<m$.
We prove that $W$ has a minimal additive complement if $V'$ satisfies
the big-hole property modulo $m$. This answers a problem posed by
Ma and Chen.
\end{abstract}



\section{Introduction}
Let $\mathbb Z$ be the set of integers and put
\[
 \mathbb N_0=\{0,1,2,\ldots\},\qquad
 H+W=\{h+w:h\in H,\ w\in W\}.
\]
A set $H\subseteq\mathbb Z$ is an additive complement to a nonempty
set $W\subseteq\mathbb Z$ if
\begin{equation*}
 H+W=\mathbb Z.
\end{equation*}
It is a minimal additive complement if
\begin{equation*}
 H+W=\mathbb Z,\qquad
 (H\setminus\{h\})+W\ne\mathbb Z\quad(h\in H).
\end{equation*}

In 2011, Nathanson~\cite{N11} asked whether every infinite set of
integers has a minimal additive complement. In 2012, Chen and
Yang~\cite{CY12} proved that such a complement exists in each of the
following cases:
\[
 \text{(i)}\quad \inf W=-\infty\quad\text{and}\quad \sup W=+\infty;
\]
\[
 \text{(ii)}\quad W=\{w_1<w_2<\cdots\},\qquad
 \limsup_{j\to\infty}(w_{j+1}-w_j)=\infty.
\]

To state the subsequent results, let $m\geq1$ be an integer. For
$a\in\mathbb Z$ and $S\subseteq\mathbb Z$, put
\[
 [m]=\{0,1,\ldots,m-1\},\qquad
 [a]_m\in[m],\quad a\equiv[a]_m\pmod m,\qquad
 [S]_m=\{[s]_m:s\in S\}.
\]
Suppose that $W$ is bounded below and
\begin{equation}\label{eq3}
 \begin{gathered}
 W=(m\mathbb N_0+X)\cup U\cup V,\\
 [U]_m\subseteq[X]_m,\qquad [X]_m\cap[V]_m=\varnothing,
 \end{gathered}
\end{equation}
where $V$ contains no infinite arithmetic progression with common
difference $m$. We say that $W$ satisfies the $\mathcal C$-condition
modulo $m$ if $X\ne\varnothing$ and there exists
$\mathcal C\subseteq[m]$ with
\begin{equation}\label{eq4}
 [\mathcal C+(X\cup V)]_m=[m]
\end{equation}
such that, for each $c\in\mathcal C$, some $v_c\in V$ satisfies
\begin{equation}\label{eq5}
 [c+v_c]_m\notin
 [(\mathcal C\setminus\{c\})+(X\cup V)]_m.
\end{equation}

In 2019, Kiss, S\'andor and Yang~\cite{KSY19} proved the following
result.\\
{\bf Theorem A.} Let $W$ satisfy \eqref{eq3}--\eqref{eq5}.
If $0<|V|<\infty$, then there exists a minimal additive complement to $W$.

In the same paper, Kiss, S\'andor and Yang~\cite{KSY19} asked whether
the finiteness assumption on $V$ in Theorem A can be removed.
Also in 2019, Kwon~\cite{K19} obtained the following result.\\
{\bf Theorem B.} Let $W$ satisfy \eqref{eq3}--\eqref{eq5}.
If $|[V]_m|=1$ and $V$ has a minimal additive complement,
then there exists a minimal additive complement to $W$.

For related results on minimal additive complements, see
\cite{CF20,BSA20,AKL21,BSG21,BSC21,BSN21,BL23,Z23,GT25,CD26,CT26,GT26,HT26,P26,TH26}.

In 2026, Ma and Chen~\cite{MC} considered the following condition.

Let $Y=\{y_1<y_2<\cdots\}\subseteq\mathbb Z$ be bounded below, and put
\[
 Y_r=Y\cap(r+m\mathbb Z)\qquad(r\in[m]).
\]
We say that $Y$ has the big-hole property modulo $m$ if
\begin{equation}\label{eq6}
 |Y_r|=\infty\qquad(r\in[Y]_m)
\end{equation}
and
\begin{equation}\label{eq7}
 \begin{aligned}
 \limsup_{\substack{j\to\infty\\y_j\in Y_r}}
 (y_j-y_{j-1})&=\infty,\\
 \limsup_{\substack{j\to\infty\\y_j\in Y_r}}
 (y_{j+1}-y_j)&=\infty
 \end{aligned}
 \qquad(r\in[Y]_m).
\end{equation}
Here $y_{j-1},y_j,y_{j+1}$ are consecutive members of $Y$;
$y_{j-1}$ and $y_{j+1}$ need not belong to $Y_r$.
For $Y=\varnothing$, \eqref{eq6}--\eqref{eq7} are understood to
hold vacuously.

Ma and Chen~\cite{MC} proved the following result.\\
{\bf Theorem C.} Let $W$ satisfy \eqref{eq3}--\eqref{eq5}.
If $V$ satisfies the big-hole property modulo $m$, then
there exists a minimal additive complement to $W$.

For $r\in[m]$, put
\begin{equation*}
 V_r=V\cap(r+m\mathbb Z),\qquad
 V'=\bigcup_{\substack{r\in[m]\\|V_r|=\infty}}V_r.
\end{equation*}
In the same paper, Ma and Chen~\cite[Problem~1.2]{MC} posed the following problem.\\
{\bf Ma and Chen's Problem.} Let $W$ satisfy \eqref{eq3}--\eqref{eq5}.
If $V'$ satisfies the big-hole property modulo $m$,
does there exist a minimal additive complement to $W$?

In 2026, Tang and He~\cite{TH26} gave an affirmative answer to Problem~1.3 of
Ma and Chen, which concerns not eventually periodic $S$-difference sets.
Here we consider Problem~1.2.

The following theorem gives an affirmative answer to the problem of
Ma and Chen.

\begin{theorem}\label{thm1}
Let $W$ satisfy \eqref{eq3}--\eqref{eq5}, and define $V'$ as above. If $V'$ satisfies the big-hole property modulo $m$, then
there exists a minimal additive complement to $W$.
\end{theorem}

\section{Preliminary Lemmas}

\begin{lemma}\label{lem1}
Let $B$ be a countable set, and let $\mathcal K$ be a family of
nonempty finite subsets of $B$. Suppose that
\begin{equation*}
 \bigl|\{K\in\mathcal K:b\in K\}\bigr|<\infty
 \qquad(b\in B)
\end{equation*}
and
\begin{equation*}
 K\subseteq K'\quad\text{only if}\quad K=K'
 \qquad(K,K'\in\mathcal K).
\end{equation*}
Put
\[
 B_0=\bigcup_{K\in\mathcal K}K.
\]
For every countable family $\mathcal E$ of infinite subsets of $B_0$,
there exists $H\subseteq B_0$ such that
\begin{equation*}
 \begin{aligned}
 H\cap K&\ne\varnothing &&(K\in\mathcal K),\\
 H\cap E&\ne\varnothing &&(E\in\mathcal E),
 \end{aligned}
\end{equation*}
and, for every $h\in H$, there exists $K_h\in\mathcal K$ with $H\cap K_h=\{h\}$.
\end{lemma}

\begin{proof}
If $\mathcal K=\varnothing$, then
\[
 B_0=\varnothing,\qquad \mathcal E=\varnothing,\qquad
 H=\varnothing
\]
satisfies the conclusions of the lemma.
Suppose $\mathcal K\ne\varnothing$.
For every finite set $A\subseteq B_0$, put
\[
 N(A)=\bigcup_{\substack{K\in\mathcal K\\K\cap A\ne\varnothing}}K.
\]
By the assumption that each $b\in B$ lies in only finitely many members of $\mathcal K$,
\begin{equation*}
 \begin{aligned}
 \bigl|\{K\in\mathcal K:K\cap A\ne\varnothing\}\bigr|
 &\leq\sum_{b\in A}|\{K\in\mathcal K:b\in K\}|<\infty,\\
 |N(A)|&<\infty,\qquad |N(N(A))|<\infty
 \end{aligned}
 \qquad(|A|<\infty).
\end{equation*}
Enumerate $\mathcal E$ as $E_1,E_2,\ldots$, allowing a finite
or empty list, and put $A_0=\varnothing$.
For each listed index $i$, suppose that $A_{i-1}$ has been defined
and is finite. Then
\[
 |N(N(A_{i-1}))|<\infty,\qquad |E_i|=\infty.
\]
Choose
\[
 v_i\in E_i\setminus N(N(A_{i-1})).
\]
Since $v_i\in E_i\subseteq B_0$, choose $K_i\in\mathcal K$ with
$v_i\in K_i$, and put
\[
 A_i=A_{i-1}\cup K_i=\bigcup_{j\leq i}K_j.
\]
Since $K_i$ is finite, $A_i$ is finite.
Fix $j<i$ and $K\in\mathcal K$. If $K\cap K_j\ne\varnothing$,
then $K\cap A_{i-1}\ne\varnothing$, $K\subseteq N(A_{i-1})$.

If also $K\cap K_i\ne\varnothing$, then
\[
 \begin{aligned}
 K_i\cap N(A_{i-1})\ne\varnothing,~
 K_i\subseteq N(N(A_{i-1})),~
 v_i\in N(N(A_{i-1})),
 \end{aligned}
\]
contrary to the choice of $v_i$. Therefore
\begin{equation*}
 \bigl|\{i:K\cap K_i\ne\varnothing\}\bigr|\leq1
 \quad(K\in\mathcal K),\qquad
 K_i\cap K_j=\varnothing\quad(i\ne j).
\end{equation*}
Put
\begin{equation*}
 Q=\bigcup_i(K_i\setminus\{v_i\}),\qquad H^*=B_0\setminus Q.
\end{equation*}
Suppose that $K\in\mathcal K$ and $K\cap H^*=\varnothing$.
By the definitions of $Q$ and $H^*$,
\[
 \varnothing\ne K\subseteq Q
 =\bigcup_i(K_i\setminus\{v_i\}).
\]
Using the fact that each $K\in\mathcal K$ meets at most one $K_i$, there is an index $i$ for which
\[
 K\cap K_i\ne\varnothing,\qquad
 K\cap K_j=\varnothing\quad(j\ne i),
\]
and hence
\[
 K\subseteq K_i\setminus\{v_i\}\subsetneq K_i,
\]
contrary to the noncontainment assumption on $\mathcal K$. Thus
\begin{equation*}
 H^*\cap K\ne\varnothing\quad(K\in\mathcal K),\qquad
 H^*\cap K_i=\{v_i\}\quad(i\text{ listed}).
\end{equation*}

Enumerate $H^*=\{b_1,b_2,\ldots\}$, allowing a finite or empty
list, and set $D_0=H^*$. For every listed $j$, put
\begin{equation*}
 D_j=\begin{cases}
 D_{j-1}\setminus\{b_j\},&
 \displaystyle
 (D_{j-1}\setminus\{b_j\})\cap K\ne\varnothing
       \quad(K\in\mathcal K),\\[1mm]
 D_{j-1},&\text{otherwise}.
 \end{cases}
\end{equation*}
In each case,
\begin{equation*}
 D_j\subseteq D_{j-1},\qquad
 D_j\cap K\ne\varnothing\quad(K\in\mathcal K).
\end{equation*}
Define
\[
 H=\bigcap_{j\geq0}D_j,
\]
where $D_j$ is continued constantly after the last index when the
list is finite. For each finite $K\in\mathcal K$,
\[
 K\cap H=\varnothing
\]
would imply
\[
 K\cap D_j=\varnothing
 \quad\text{for some }j\geq0,
\]
contrary to the nonempty intersections $D_j\cap K$. Hence
\begin{equation*}
 H\cap K\ne\varnothing\qquad(K\in\mathcal K).
\end{equation*}
Fix $h\in H$, and choose a listed index $j$ such that $h=b_j$.
Since
\[
 b_j\in H\subseteq D_j,
\]
we have $D_j=D_{j-1}$. Otherwise,
\[
 D_j=D_{j-1}\setminus\{b_j\},
\]
which implies $b_j\notin H$, a contradiction.
Hence there exists $K\in\mathcal K$ such that
\[
 (D_{j-1}\setminus\{h\})\cap K=\varnothing,\qquad
 D_{j-1}\cap K=\{h\}.
\]
Using the inclusions $H\subseteq D_{j-1}$ and the nonempty intersections $H\cap K$,
\[
 \{h\}\subseteq H\cap K\subseteq D_{j-1}\cap K=\{h\},
\qquad H\cap K=\{h\}.
\]
For every listed $i$, the equality $H^*\cap K_i=\{v_i\}$ and the nonempty intersections $H\cap K$ give
\[
 \varnothing\ne H\cap K_i
 \subseteq H^*\cap K_i=\{v_i\},
\]
and hence
\[
 v_i\in H\cap K_i,\qquad
 v_i\in E_i,\qquad H\cap E_i\ne\varnothing.
\]
This proves the lemma.
\end{proof}

\begin{lemma}\label{lem2}
Let $m\geq1$ and let $Y\subseteq\mathbb Z$ be bounded below and
satisfy \eqref{eq6}--\eqref{eq7} modulo $m$. Let
$F\subseteq\mathbb Z$ be finite. Put
\[
 V=Y\cup F,\qquad B=D+m\mathbb Z
 \quad(D\subseteq[m]).
\]
For every $T\subseteq B+Y$, there exists $H\subseteq B$ such that
\begin{equation*}
 T\subseteq H+V,\qquad
 T\not\subseteq(H\setminus\{h\})+V\quad(h\in H).
\end{equation*}
\end{lemma}

\begin{proof}
For $T=\varnothing$, take $H=\varnothing$. Suppose $T\ne\varnothing$.
Choose
\begin{equation*}
 M>\max(F\cup\{0\}).
\end{equation*}
We construct distinct $n_1,n_2,\ldots\in T$ and
$h_1,h_2,\ldots\in B$, with
\begin{equation*}
 0>h_1>h_2>\cdots,
\end{equation*}
such that
\begin{equation}\label{eq24}
 n_j-h_i\in V\quad\text{if and only if}\quad i=j
 \qquad(i,j\text{ already chosen}).
\end{equation}
At the $j$-th step choose $n_j$ of minimum $|n_j|$ in
\begin{equation}\label{eq25}
 T\setminus\bigl(\{h_1,\ldots,h_{j-1}\}+V\bigr),
\end{equation}
provided this set is nonempty; otherwise stop.
In the first step,
\[
 n_1\in T\subseteq B+Y,
\]
so there are $r\in[Y]_m$ and $d\in D$ with
\[
 [n_1-r]_m=d.
\]
By \eqref{eq6}, choose $y\in Y_r$ with
\[
 y>n_1,\qquad h_1=n_1-y<0.
\]
Then
\[
 [h_1]_m=[n_1-r]_m=d\in D,\qquad
 h_1\in B,\qquad n_1-h_1=y\in V.
\]
Thus the required conditions hold for the first step.

Suppose $n_i,h_i$ have been selected for $1\leq i\leq k$.
If \eqref{eq25} is nonempty for $j=k+1$, put $n=n_{k+1}$.
For $1\leq i\leq k$, the set from which $n$ is chosen is a subset
of the set from which $n_i$ was chosen. Moreover,
$n_i\in h_i+V$ and $n\notin\{h_1,\ldots,h_k\}+V$.
By the minimum rule, we have
\begin{equation*}
 |n_i|\leq|n|,\qquad n_i\ne n
 \qquad(1\leq i\leq k).
\end{equation*}
Moreover,
\begin{equation*}
 n-h_i\notin V\qquad(1\leq i\leq k).
\end{equation*}
If $0\in T$, the minimum rule \eqref{eq25} gives
\[
 |n_1|=\min\{|t|:t\in T\}=0,\qquad n_1=0.
\]
Since the $n_i$ are distinct,
\[
 n_{k+1}\ne0\qquad(k\geq1).
\]
The same conclusion holds if $0\notin T$.
Since $n\in B+Y$, choose $r\in[Y]_m$ satisfying
\begin{equation*}
 [n-r]_m\in D.
\end{equation*}

\medskip
\noindent\textbf{Case 1.} \textit{$n>0$.}
By \eqref{eq7}, choose $y_u\in Y_r$ so that
\begin{equation*}
 y_u>n-h_k,\qquad y_{u-1}>M,\qquad
 y_u-y_{u-1}>2n.
\end{equation*}
Put $h_{k+1}=n-y_u$. Then the inequalities $0>h_1>\cdots>h_k$ and the choices of $r$ and $y_u$
give
\[
 [h_{k+1}]_m=[n-r]_m\in D,\qquad
 h_{k+1}\in B,\qquad
 h_{k+1}<h_k<0,
\]
and
\[
 n-h_{k+1}=y_u\in Y_r\subseteq V.
\]
From the minimum rule \eqref{eq25},
\[
 -n\leq n_i\leq n,\qquad n_i\ne n,
\]
whence
\[
 -2n\leq n_i-n<0\qquad(1\leq i\leq k).
\]
Therefore
\begin{equation*}
 \begin{aligned}
 M <y_{u-1}<y_u-2n \leq y_u+n_i-n=n_i-h_{k+1}<y_u
 \qquad(1\leq i\leq k).
 \end{aligned}
\end{equation*}
Since $y_{u-1},y_u$ are consecutive elements of $Y$ and
$F\subseteq(-\infty,M)$,
\[
 (y_{u-1},y_u)\cap Y=\varnothing,\qquad
 (M,\infty)\cap F=\varnothing.
\]
Together with the bounds for $n_i-h_{k+1}$,
\begin{equation*}
 n_i-h_{k+1}\notin Y\cup F=V
 \qquad(1\leq i\leq k).
\end{equation*}

\medskip
\noindent\textbf{Case 2.} \textit{$n<0$.}
By \eqref{eq7}, choose $y_u\in Y_r$ with
\begin{equation*}
 y_u>\max\{M,n-h_k\},\qquad
 y_{u+1}-y_u>-2n.
\end{equation*}
Put $h_{k+1}=n-y_u$. Then
\[
 [h_{k+1}]_m=[n-r]_m\in D,\qquad
 h_{k+1}\in B,\qquad
 h_{k+1}<h_k<0,
\]
and
\[
 n-h_{k+1}=y_u\in Y_r\subseteq V.
\]
From the minimum rule \eqref{eq25},
\[
 n\leq n_i\leq-n,\qquad n_i\ne n,
\]
whence
\[
 0<n_i-n\leq-2n\qquad(1\leq i\leq k).
\]
Consequently,
\begin{equation*}
 \begin{aligned}
 M<y_u<y_u+n_i-n=n_i-h_{k+1} \leq y_u-2n<y_{u+1}
 \qquad(1\leq i\leq k).
 \end{aligned}
\end{equation*}
Since $y_u,y_{u+1}$ are consecutive elements of $Y$,
\[
 (y_u,y_{u+1})\cap Y=\varnothing,\qquad
 (M,\infty)\cap F=\varnothing.
\]
Using the bounds for $n_i-h_{k+1}$,
\begin{equation*}
 n_i-h_{k+1}\notin Y\cup F=V
 \qquad(1\leq i\leq k).
\end{equation*}

In both cases, the preceding exclusions
yield
\[
 \begin{cases}
 n-h_{k+1}\in V,\\
 n-h_i\notin V &(1\leq i\leq k),\\
 n_i-h_{k+1}\notin V &(1\leq i\leq k),
 \end{cases}
\]
and hence \eqref{eq24} holds for all selected indices.

Let $H=\{h_j:j\text{ is selected}\}$. If the construction stops,
\eqref{eq25} gives
\[
 T\setminus(H+V)=\varnothing.
\]
If the construction is infinite, the minimum rule \eqref{eq25} and the distinctness
of the $n_j$ give
\[
 |n_1|\leq|n_2|\leq\cdots,
 \qquad \bigl|\{n_j:|n_j|\leq R\}\bigr|\leq2R+1
 \quad(R\in\mathbb N_0),
\]
and therefore
\[
 \lim_{j\to\infty}|n_j|=\infty.
\] For
$t\in T\setminus(H+V)$, we would have
\[
 t\in T\setminus\bigl(\{h_1,\ldots,h_{j-1}\}+V\bigr)
 \quad(j\geq1),
\]
and the minimum rule \eqref{eq25} would give
\[
 |n_j|\leq|t|\quad(j\geq1),
\]
contrary to $|n_j|\to\infty$. Thus in either case,
\[
 T\subseteq H+V.
\]
Finally, for each selected $j$, \eqref{eq24} gives
\[
 \begin{aligned}
 (n_j-V)\cap H=\{h_j\}, ~
 n_j\in h_j+V\subseteq H+V, ~
 n_j\notin(H\setminus\{h_j\})+V,
 \end{aligned}
\]
and hence
\[
 T\not\subseteq(H\setminus\{h_j\})+V
 \qquad(h_j\in H).
\]
This proves the lemma.
\end{proof}

\section{Proof of Theorem~\ref{thm1}}
Put
\begin{equation*}
 \begin{gathered}
 Y=V',\qquad F=V\setminus Y,\qquad B=\mathcal C+m\mathbb Z,\\
 C_1=[\mathcal C+X]_m,\qquad
 A=([m]\setminus C_1)+m\mathbb Z.
 \end{gathered}
\end{equation*}
By the definition of $V'$,
\[
 F=\bigcup_{\substack{r\in[m]\\|V_r|<\infty}}V_r.
\]
Since $|[m]|=m$,
\begin{equation*}
 |F|\leq\sum_{\substack{r\in[m]\\|V_r|<\infty}}|V_r|<\infty,
 \qquad V=Y\cup F,\qquad Y\cap F=\varnothing.
\end{equation*}
For $a\in A$, \eqref{eq4} and the definitions of $B$, $C_1$, and $A$ give
\[
 \begin{aligned}
 \relax[a]_m&\in[m]\setminus C_1,\\
 [a]_m&\in[\mathcal C+(X\cup V)]_m,\\
 [a]_m&\in[\mathcal C+V]_m.
 \end{aligned}
\]
Choose $c\in\mathcal C$, $v\in V$ with
\[
 [a]_m=[c+v]_m,\qquad a-v\in c+m\mathbb Z\subseteq B.
\]
Thus $a=(a-v)+v\in B+V$. Moreover, by \eqref{eq3},
\[
 \begin{aligned}
 W\setminus V&\subseteq(m\mathbb N_0+X)\cup U,\\
 [W\setminus V]_m&\subseteq[X]_m,\\
 [B+(W\setminus V)]_m&\subseteq[\mathcal C+X]_m=C_1.
 \end{aligned}
\]
Consequently,
\begin{equation}\label{eq37}
 \begin{aligned}
 A&\subseteq B+V,~~
 B+(W\setminus V) \subseteq C_1+m\mathbb Z,~~
 A\cap(C_1+m\mathbb Z) =\varnothing.
 \end{aligned}
\end{equation}

Put
\begin{equation*}
 A_0=A\setminus(B+Y).
\end{equation*}
By the definitions of $B$, $C_1$, and $A$ and $B+m=B$,
\[
 (B+Y)+m=B+Y,\qquad A+m=A.
\]
Hence
\begin{equation*}
 A_0+m=A_0,\qquad B+m=B,\qquad A+m=A.
\end{equation*}
For $n\in A_0$, define
\begin{equation*}
 E_n=(n-V)\cap B=(n-F)\cap B.
\end{equation*}
The second equality holds since
\[
 n\notin B+Y,\qquad (n-Y)\cap B=\varnothing.
\]
For $n\in A_0$, \eqref{eq37} and the definition of $A_0$ give
\[
 \varnothing\ne(n-V)\cap B.
\]
Together with the finiteness of $F$ and the translation identities for $A_0$, $B$, and $A$,
\[
 \begin{aligned}
 E_n&=(n-V)\cap B=(n-F)\cap B,\\
 0&<|E_n|\leq|F|,\\
 E_{n+m}&=((n+m)-F)\cap B=E_n+m.
 \end{aligned}
\]
For $b\in B$,
\[
 \{n\in A_0:b\in E_n\}\subseteq(b+F)\cap A_0,
\]
so that
\begin{equation*}
 \begin{aligned}
 &0<|E_n|\leq|F|\qquad(n\in A_0),\\
 &E_{n+m}=E_n+m\qquad(n\in A_0),\\
 &\bigl|\{n\in A_0:b\in E_n\}\bigr|\leq|F|
       \qquad(b\in B).
 \end{aligned}
\end{equation*}

Define the family of inclusion-minimal sets
\begin{equation*}
 \mathcal K=\bigl\{E_n:n\in A_0,\
 \text{there is no }n'\in A_0\text{ with }E_{n'}\subsetneq E_n\bigr\}.
\end{equation*}
For each $n\in A_0$, the nonempty family $\{E_{n'}:n'\in A_0,\ E_{n'}\subseteq E_n\}$
contains a member of minimum cardinality. Thus, for each $n\in A_0$, there exists $K\in\mathcal K$ such that $K\subseteq E_n$.

By the translation identities for $A_0$, $B$, and $A$ and the properties of $E_n$ just established,
\[
 E_{n+m}=E_n+m,\qquad
 E_{n-m}=E_n-m\qquad(n\in A_0).
\]
Thus the definition of $\mathcal K$ gives
\[
 K\in\mathcal K\quad\text{if and only if}\quad
 K+m\in\mathcal K.
\]
For each $K\in\mathcal K$ with $b\in K$, select one
$n_K\in A_0$ such that $K=E_{n_K}$. Then
\[
 n_K\in b+F,\qquad
 K\ne K'\quad\text{only if}\quad n_K\ne n_{K'}.
\]
Since $F$ is finite, the properties of $E_n$ and the definition of $\mathcal K$ give
\begin{equation*}
 \begin{gathered}
 K,K'\in\mathcal K,\quad K\subseteq K'
       \quad\text{only if}\quad K=K',\\
 |\{K\in\mathcal K:b\in K\}|\leq|F|\quad(b\in B),\\
 \{K+m:K\in\mathcal K\}=\mathcal K.
 \end{gathered}
\end{equation*}
Put
\begin{equation*}
 B_f=\bigcup_{K\in\mathcal K}K,\qquad B_i=B\setminus B_f.
\end{equation*}
Thus
\begin{equation*}
 B_f+m=B_f,\qquad B_i+m=B_i,\qquad B=B_f\cup B_i,
 \qquad B_f\cap B_i=\varnothing.
\end{equation*}
Fix $n\in A$ with $(n-Y)\cap B_f\ne\varnothing$.
Choose
\[
 b=n-y\in B_f,\qquad y\in Y_r\quad(r=[y]_m).
\]
For every $y'\in Y_r$,
\[
 \begin{aligned}
 y'&\equiv y\pmod m,\\
 n-y'&\equiv n-y=b\pmod m,\\
 n-y'&\in b+m\mathbb Z\subseteq B_f
 \end{aligned}
\]
by the definitions and translation invariance of $B_f$ and $B_i$. Hence
\[
 n-Y_r\subseteq(n-Y)\cap B_f,\qquad
 |n-Y_r|=|Y_r|=\infty
\]
by \eqref{eq6}, giving
\begin{equation*}
 |(n-Y)\cap B_f|=\infty
 \quad\text{whenever }(n-Y)\cap B_f\ne\varnothing.
\end{equation*}

Put
\[
 \mathcal E=\{(n-Y)\cap B_f:n\in A,
 \ (n-Y)\cap B_f\ne\varnothing\}.
\]
Since $A,B_f\subseteq\mathbb Z$ and each member of
$\mathcal E$ is determined by some $n\in A$, the sets $A$,
$B_f$ and $\mathcal E$ are all at most countable.
By the properties of $\mathcal K$ and the infinitude of the nonempty sets $(n-Y)\cap B_f$, the assumptions of
Lemma~\ref{lem1} hold. Thus there exists $H_0\subseteq B_f$ with
\begin{equation*}
 H_0\cap K\ne\varnothing\quad(K\in\mathcal K),\qquad
 H_0\cap((n-Y)\cap B_f)\ne\varnothing
 \quad\text{if }(n-Y)\cap B_f\ne\varnothing,
\end{equation*}
and, for each $h\in H_0$, there exists $K\in\mathcal K$ with
$H_0\cap K=\{h\}$.

For $\mathcal K=\varnothing$, we have
\[
 B_f=\varnothing,\qquad\mathcal E=\varnothing,
 \qquad H_0=\varnothing,
\]
and these intersection properties hold vacuously.

For $n\in A_0$, choose $K\in\mathcal K$ with
$K\subseteq E_n$ by the choice of $\mathcal K$. Then the intersection property of $H_0$ gives
\[
 \varnothing\ne H_0\cap K\subseteq H_0\cap E_n,
 \qquad n\in H_0+V.
\]
For $n\in A\cap(B_f+Y)$, the intersection property of $H_0$ gives
\[
 \varnothing\ne H_0\cap((n-Y)\cap B_f)
 \subseteq H_0\cap(n-Y),
 \qquad n\in H_0+Y\subseteq H_0+V.
\]
Therefore
\begin{equation*}
 A_0\cup\bigl(A\cap(B_f+Y)\bigr)\subseteq H_0+V.
\end{equation*}

Fix $h\in H_0$. By the definition of $\mathcal K$ and the single-point intersection property of $H_0$, there is
$n_h\in A_0$ such that
\[
 K=E_{n_h},\qquad H_0\cap E_{n_h}=\{h\}.
\]
Since $n_h\in A_0$ and $E_{n_h}\subseteq B_f$,
\begin{equation*}
 \begin{aligned}
 (n_h-V)\cap H_0&=E_{n_h}\cap H_0=\{h\},\\
 (n_h-V)\cap B_i&=E_{n_h}\cap B_i=\varnothing.
 \end{aligned}
\end{equation*}

Next let
\begin{equation*}
 T=A\setminus(H_0+V).
\end{equation*}
By the definition of $A_0$ and the inclusion just proved,
\[
 \begin{aligned}
 T \subseteq A\setminus\bigl(A_0\cup(B_f+Y)\bigr) \subseteq(B+Y)\setminus(B_f+Y).
 \end{aligned}
\]
For $n\in(B+Y)\setminus(B_f+Y)$, choose
\[
 n=b+y,\qquad b\in B,\quad y\in Y.
\]
Then $b\notin B_f$ and $b\in B_i$ by the definitions and translation invariance of $B_f$ and $B_i$. Thus
\begin{equation*}
 T\subseteq(B+Y)\setminus(B_f+Y)\subseteq B_i+Y.
\end{equation*}
By the definitions and translation invariance of $B_f$ and $B_i$, put
\[
 D_i=[B_i]_m\subseteq[m],\qquad B_i=D_i+m\mathbb Z.
\]
Applying Lemma~\ref{lem2} to $Y,F,B_i$, and $T$, we obtain
$H_1\subseteq B_i$ such that
\begin{equation}\label{eq54}
 T\subseteq H_1+V,\qquad
 T\not\subseteq(H_1\setminus\{h\})+V\quad(h\in H_1).
\end{equation}
Put
\begin{equation*}
 H=H_0\cup H_1\subseteq B.
\end{equation*}
From the definition of $T$ and \eqref{eq54},
\begin{equation*}
 A\subseteq H+V.
\end{equation*}
For $h\in H_0$, the identities for $(n_h-V)\cap H_0$ and $(n_h-V)\cap B_i$ and $H_1\subseteq B_i$ give
\[
 \begin{aligned}
 (n_h-V)\cap H
 &=\bigl((n_h-V)\cap H_0\bigr)
 \cup\bigl((n_h-V)\cap H_1\bigr) =\{h\}\cup\varnothing=\{h\}.
 \end{aligned}
\]
Consequently,
\begin{equation*}
 (n_h-V)\cap H=\{h\},\qquad
 n_h\notin(H\setminus\{h\})+V.
\end{equation*}
For $h\in H_1$, \eqref{eq54} gives an integer $n_h$ with
\[
 \begin{aligned}
 n_h \in T\setminus\bigl((H_1\setminus\{h\})+V\bigr),~
 n_h \in H_1+V,~
 n_h \in h+V.
 \end{aligned}
\]
By the definition of $T$,
\[
 n_h\notin H_0+V,\qquad
 n_h\notin(H_1\setminus\{h\})+V,
\]
and therefore
\begin{equation*}
 n_h\notin\bigl(H_0\cup(H_1\setminus\{h\})\bigr)+V
 =(H\setminus\{h\})+V.
\end{equation*}
Thus the preceding inclusions and exclusions show that, for each $h\in H$,
there exists $n_h\in A$ such that
\begin{equation}\label{eq59}
 A\subseteq H+V,\qquad
 n_h\notin(H\setminus\{h\})+V.
\end{equation}

For $c\in\mathcal C$, choose $v_c\in V$ satisfying \eqref{eq5},
and put
\begin{equation*}
 a_c=[c+v_c]_m,\qquad H_c=H\cap(c+m\mathbb Z).
\end{equation*}
Suppose $a_c\in C_1$. Then for some $c'\in\mathcal C$, $x\in X$,
\[
 a_c=[c+v_c]_m=[c'+x]_m.
\]
If $c'\ne c$, then
\[
 [c+v_c]_m\in
 [(\mathcal C\setminus\{c\})+X]_m
 \subseteq[(\mathcal C\setminus\{c\})+(X\cup V)]_m,
\]
contrary to \eqref{eq5}. If $c'=c$, then
\[
 [v_c]_m=[x]_m\in[X]_m\cap[V]_m,
\]
contrary to \eqref{eq3}. Hence
\begin{equation*}
 a_c\notin C_1,\qquad a_c+m\mathbb Z\subseteq A.
\end{equation*}
Let $z\in a_c+m\mathbb Z$. By \eqref{eq59}, $z=h+v\quad\text{for some }h\in H,\ v\in V$.

Since $H\subseteq\mathcal C+m\mathbb Z$, there exists
$c'\in\mathcal C$ such that
\[
 [h]_m=c',\qquad
 [c+v_c]_m=a_c=[h+v]_m=[c'+v]_m.
\]
If $c'\ne c$, \eqref{eq5} would be contradicted. Thus
$c'=c$, $h\in H_c$, and therefore
\begin{equation*}
 a_c+m\mathbb Z\subseteq H_c+V
 \qquad(c\in\mathcal C).
\end{equation*}
By \eqref{eq3}, $V\subseteq W$ is bounded below. Choose
$L\in\mathbb Z$ such that $ v\geq L~(v\in V)$.

For each $k\geq1$, the inclusion $a_c+m\mathbb Z\subseteq H_c+V$ gives $h_k\in H_c$ and
$v_k\in V$ with $a_c-km=h_k+v_k$ and
$$h_k=a_c-km-v_k \leq a_c-km-L\rightarrow-\infty
 \qquad(k\to\infty).$$
Hence $\inf H_c=-\infty~(c\in\mathcal C)$.

Let $z\in C_1+m\mathbb Z$. There are $c\in\mathcal C$, $x\in X$
with $z\equiv c+x\pmod m$. By the equality $\inf H_c=-\infty$, choose $h\in H_c$ such that $h\leq z-x$.
Consequently,
\begin{equation*}
 \frac{z-h-x}{m}\in\mathbb N_0,\qquad
 z-h\in x+m\mathbb N_0\subseteq W.
\end{equation*}
It follows that $C_1+m\mathbb Z\subseteq H+W$.

By \eqref{eq59}, the inclusion $C_1+m\mathbb Z\subseteq H+W$, and the definitions of $B$, $C_1$, and $A$,
\[
 \begin{aligned}
 \mathbb Z
 =A\cup(C_1+m\mathbb Z)
 \subseteq(H+V)\cup(H+W)
  \subseteq H+W\subseteq\mathbb Z.
 \end{aligned}
\]
Hence $H+W=\mathbb Z$.

Fix $h\in H$ and choose $n_h\in A$ as in \eqref{eq59}.
By \eqref{eq37},
$$(H\setminus\{h\})+(W\setminus V)
  \subseteq B+(W\setminus V)
  \subseteq C_1+m\mathbb Z,
 n_h \in A,$$
 $$A\cap(C_1+m\mathbb Z)=\varnothing.$$
Together with \eqref{eq59},
\begin{equation*}
 \begin{aligned}
 n_h \notin(H\setminus\{h\})+V, ~
 n_h \notin(H\setminus\{h\})+(W\setminus V), ~
 n_h \notin(H\setminus\{h\})+W.
 \end{aligned}
\end{equation*}
Therefore, the preceding conclusions give
\[
 H+W=\mathbb Z,\qquad
 (H\setminus\{h\})+W\ne\mathbb Z\quad(h\in H).
\]

Therefore, there exists a minimal additive complement $H$ to $W$.

\end{document}